\documentclass[a4paper,12pt]{amsart}
\usepackage{fullpage}
\usepackage{amsmath}
\usepackage{amsfonts}
\usepackage{amssymb}
\usepackage{color}
\usepackage[OT4]{fontenc}
\usepackage[latin2]{inputenc}

\DeclareMathOperator\CC{{\it\mathbb C^n}}

\DeclareMathOperator\tiu{\it{\tilde{u}}}

\DeclareMathOperator\la{{\lambda}}
\DeclareMathOperator\ep{{\varepsilon}}

\DeclareMathOperator \lbr {{\it\lbrace}}
\DeclareMathOperator \rbr {{\it\rbrace}}

\DeclareMathOperator\pa{\partial}

\def\al{\alpha}

\newtheorem{theorem}{Theorem}

\newtheorem{definition}[theorem]{Definition}
\newtheorem{corollary}[theorem]{Corollary}
\newtheorem{proposition}[theorem]{Proposition}
\newtheorem{example}[theorem]{Example}
\newtheorem{remark}[theorem]{Remark}
\title{The minimum sets and free boundaries of strictly plurisubharmonic functions}

\author{S\l awomir Dinew}
\address{S\l awomir Dinew\\
Department of Mathematics and Computer Science\\
Jagiellonian University, Poland}\email{slawomir.dinew@im.uj.edu.pl}
\author{\.Zywomir  Dinew}
\address{\.Zywomir Dinew\\
Department of Mathematics and Computer Science\\
Jagiellonian University, Poland} \email{zywomir.dinew@im.uj.edu.pl}

\begin{document}
 
\begin{abstract}
We study the minimum sets of plurisubharmonic functions with strictly positive Monge-Amp\`ere densities. We investigate the relationship between their Hausdorff dimension and the regularity of the function. Under suitable assumptions we prove that the minimum set cannot contain analytic subvarieties of large dimension. In the planar case we analyze the influence on the regularity of the right hand side and consider the corresponding free boundary problem with irregular data. We provide sharp examples for the Hausdorff dimension of the minimum set and the related free boundary. We also draw several analogues with the corresponding real results.
\end{abstract}
\maketitle
\section{Introduction}

A classical theorem of Harvey and Wells \cite{HW73} states that the zero set of a nonnegative strictly plurisubharmonic and smooth function is contained in a $\mathcal C^1$ totally real hypersurface. In particular this implies that the Hausdorff dimension of the zero set is {\it small} compared to the dimension of the ambient space, and the zero set has {\it no analytic structure}.

There are many good reasons to study such minimum sets. One of them is that compact pieces of such satisfy the Condition  (P) introduced by Catlin in \cite{Cat84} which is crucial for the compactness of the $\overline{\partial}$-Neumann problem. In a completely different direction El Mir \cite{EM84} has shown that zero sets of bounded continuous strictly plurisubharmonic functions are removable sets in the theory of extensions of closed positive currents. In both settings it is crucial that the function is strictly plurisubharmonic.

 Our motivation for the investigation of generalizations of such minimum sets comes from the study of compactness properties of
 solutions to the complex
 Monge-Amp\`ere
 equation. Analogous theory for the real Monge-Amp\`ere equation was developed by Caffarelli \cite{Caf89,Caf90b} and the analysis of the
 corresponding minimal sets is crucial there. 


The real counterpart of the theory, with plurisubharmonic functions replaced by {\it convex} ones, is trivial for the minimum set of a smooth strictly convex function is always a singleton. If strict convexity is replaced simply by convexity the zero set can be any preassigned convex set. On the other hand, when strict convexity is relaxed to strict positivity or the real Monge-Amp\`ere operator, the picture is drastically different. In fact understanding how a convex solution to a Monge-Amp\`ere equation with strictly positive right hand side {\it may fail to be strictly convex} is the heart of the matter of the Caffarelli regularity theory  (see \cite{Caf90b,Gut01}). As classical examples of Pogorelov \cite{Pog71}  (see also Example \ref{generalizedpogorelov} below) show, the minimum set in this case can be a line or a lower dimensional piece of linear subspace. Its Hausdorff dimension can be estimated  (\cite{Caf93, Mo15}), and, as we shall see below, it is strictly related to the {\it regularity} of the function itself. Our first observation is as follows:
\begin{proposition} Let $v$ be a nonnegative convex function in a domain $\Omega$,\ $\left (\Omega\Subset\mathbb R^n\right)$ satisfying $Det\left (D^2v\right)\geq C>0$  (the inequality is to be understood in the
 viscosity sense). Assume moreover that $v\in\mathcal C^{1,\alpha}$ for $\alpha>1-\frac{2k}n$. Then one has the Hausdorff dimension estimate
$$dim_{\mathcal H}\lbr v^{-1}\left (0\right)\rbr< k.$$
 
\end{proposition}

Returning to the complex realm, if strict plurisubharmonicity is exchanged to mere plurisubharmonicity, then there is almost no control of the minimum set. In fact every regular compact set $K$ in $\mathbb C^n$ (see \cite{Kli91} for a definition) is the zero set of the  (nonnegative) global extremal plurisubharmonic function associated to $K$. It is nevertheless interesting to consider the intermediate condition: we investigate nonnegative plurisubharmonic functions for which the complex Monge-Amp\`ere operator is {\it strictly positive}.

As observed by B\l ocki \cite{Bl99}, Pogorelov examples from \cite{Pog71} easily generalize to the complex setting  (the important difference being that, unlike the real case, there is no difference between complex dimension $2$ and higher dimensions). Thus our imposed condition cannot rule out a complex analytic structure within the zero set. It is however reasonable to ask whether one can control its dimension just like the dimension of the affine set in the real case.

Our next result confirms this expectation:
\begin{theorem}
 Let $u\geq 0$ be a plurisubharmonic function in a domain $\Omega$,\ in $\mathbb C^n$, satisfying $\left (dd^c u\right)^n\geq 1$. If
 additionally $u\in \mathcal C^{1,\alpha}$ for
 $\alpha>1-\frac{2k}n$ if $2k\leq n$ or $u\in \mathcal C^{0,\beta}$ for $\beta> 2-\frac{2k}n$ if $2k>n$, then no
 analytic set of dimension $\geq k$ can be contained in $u^{-1}\left (0\right)$.
\end{theorem}

Such a theorem may find applications in the study of the local regularity of the complex Monge-Amp\`ere equation. Indeed the result is yet another evidence that $\mathcal C^{1,\beta}$ functions  ($\beta>1-2/n$) with strongly positive and $\gamma$-H\"older continuous Monge-Amp\`ere density $\left (0<\gamma<1\right)$ should be classical solutions. If such a statement is true, then the Pogorelov example is ``the worst one'' in the H\"older scale and for a smooth strictly positive density any solution which is more regular should be automatically smooth. For the real Monge-Amp\`ere equation analogous theorem was proven by Urbas and Caffarelli  \cite{Ur88, Caf90b}. In the complex setting the problem is still largely open and  we refer to \cite{DZZ11,Wa13} for partial results in this direction. It should be noted that in the Sobolev scale Pogorelov examples are indeed the worst ones as the main result in \cite{BD11} shows.

On the other hand the Hausdorff dimension of the whole zero set is much harder to control. We have divided our investigation in the planar  (i.e. $n=1$) case and the multidimensional one.

When considering the planar case  we deal with strictly subharmonic functions. Such a setting sounds very classical but quite to our surprise we were unable to find much in the existent literature. On the bright side we found a lot of results in a closely related {\it free boundary problem} theory which in a sense can be thought of as a one-sided version of minimum sets  (\cite{Bla01,Caf77, Caf81}). In the free boundary problem theory the equations are usually considered for
 substantially more regular right hand sides and the main purpose is to establish additional regularity for the free boundary set. Thus
 the technical details are quite different at places.  In particular the following estimate was a strong motivation for our
 investigations  (see \cite{Caf81}):
\begin{theorem}[Caffarelli'81]
 Let $u\in\mathcal C^{1,1}\left (U\right),\ U\Subset\mathbb C$ be a nonnegative subharmonic function satisfying
$$\Delta u=f$$
on the set $u>0$ for some Lipschitz strictly positive function $f$. Then the free boundary of $u$ has locally finite $1$-dimensional Hausdorff measure. In particular its Hausdorff dimension is no more than 1.
\end{theorem}
 Blank noticed in \cite{Bla01} that it is enough to assume that $f$  (still strictly positive) is in $W^{1,p}$ for some $p>2$. In fact it
 is even enough to assume that $f\in L^{\infty}$ and $f\in W^{1,1}$. On the other hand Blank himself gave a very interesting example in
 \cite{Bla01} showing that for $f$ less than Dini smooth the free boundary can be badly behaved- in particular it can spin around a point
 infinitely often. All this suggested that the minimum sets, just like free boundaries, can be badly behaved but are of Hausdorff
 dimension less or equal to one.
  
Our next result disproves that:
\begin{theorem} In the planar case there are compact sets $K$ and $FB$, such that $K$ is a minimum set of a strictly subharmonic function and $FB$ is a free boundary such that $dim_{\mathcal H}K=dim_{\mathcal H}FB>1$.
\end{theorem}

In fact in can be checked that for any $p>1,\ \varepsilon>0$ the Laplacian density in our examples can be taken to belong to $L^p$ and to $W^{1,1-\varepsilon}$ which shows that Caffarelli theorem is fairly sharp.

In the multidimensional case the Hausdorff dimension of the minimum sets can also be larger than (the expected) $n$ as the examples \ref{juliamanydim} and \ref{generalizedpogorelov} show. The corresponding free boundary problem is also of interest - in fact
free boundary problems have been
 considered for nonlinear operators \cite{Bla01,
 KKPS00, LS03}. To our knowledge the free boundary problem for the complex Monge-Amp\`ere equation has not been thoroughly investigated
 and we plan to consider this in a future article \cite{DD}.
\section{preliminaries}
In this section we collect the definitions and basic properties of the notions that will appear later on.

{\bf Minimum sets and free boundaries.} The following definition of a minimum set will be used throughout:
\begin{definition} Let $U$ be an open set in $\CC$ and $K$ be a closed subset in $U$. Then $K$ is said to be a minimum set if there exists a plurisubharmonic function $u$ on $U$ such that $u\geq 0$ and $K\subset u^{-1}\left (0\right)$.
 \end{definition}
If $u$ satisfies further restrictions these yield constraints on $K$. In particular we will be interested in the case when $u$ satisfies the condition
\begin{equation}
 \left (i\pa\bar{\pa}u\right)^n\geq c>0
\end{equation}
 (we refer to \cite{Kli91} for the pluripotential definition of the complex Monge-Amp\`ere operator). Of course if $u\in\mathcal C^2$ this
 is equivalent to strict plurisubharmonicity but our main interest will be what happens for more singular $u$.

In the planar case we introduce a seemingly weaker notion of a {\it function strictly subharmonic at $K$}:
\begin{definition}
 A nonnegative planar subharmonic function $u$ is said to be strictly subharmonic at its minimum set $K$ if
there exists  $c>0$ such that for any $z_0\in K$ $$\liminf_{r\rightarrow 0^{+}}1/r^2\int_{B\left (z_0,r\right)}\Delta u\geq c.$$ 
\end{definition}
Intuitively this means that close to $K$, $u$ is strictly subharmonic in the average sense.

Next we define the free boundary set. As we shall consider only the planar case, we give the definition only in this setting.

Following \cite{Caf81} if $u\geq 0$ is subharmonic in a domain $\Omega$ in $\mathbb C$, such that 
\begin{enumerate}
\item on $\Omega\left (u\right)=\lbr u>0\rbr$, $u$ satisfies $\Delta u=f$ for $f\geq c>0$;
\item $u$ and $\nabla u$ vanish continuously on $\pa\Omega\left (u\right)$,
\end{enumerate}
then the {\it free boundary} of $u$ is the set $FB\left (u\right):=\pa\Omega\left (u\right)\cap\pa\lbr u=0\rbr$.
\begin{remark}
 Usually additional regularity requirements are put on $f$. In our setting we impose nothing besides strict positivity.
\end{remark}

{\bf Porosity.}
We recall the notion of a porous set, which comes in handy in establishing bounds for the Hausdorff dimension of a given set:
\begin{definition}
 Let $K$ be a compact subset of $\CC$. Given any number $\la\in\left (0,1/2\right)$ the set $K$ is called $\lambda$-porous if there exists
 $r_0>0$,
 such that for every $r$, $0<r<r_0$ and every ball $B\left (x,r\right)\subset\CC$ there exists a ball $B\left (y,\la r\right)\subset
 B\left (x,r\right)\setminus K$.
\end{definition}
 
It is a classical fact that porosity for some $\lambda$ implies bounds on the Hausdorff dimension of the set $K$  (\cite{Ma95}). The exact relationship between the optimal bound and $\lambda$ is not explicit, yet we shall only need the following simple corollary:
\begin{corollary}\label{corollary9}
 If the compact set $K\subset\CC$ is $\left (1/2-\ep\right)$-porous for some $1/2>\ep>0$, then $dim_{\mathcal H}\left (K\right)<2n$.
\end{corollary}

{\bf Green functions with pole at infinity.}
  Denote by $\mathcal L\left (\CC\right)$ the class of plurisubharmonic functions of logarithmic growth
$$\mathcal L\left (\CC\right):=\lbr u\in\mathcal{PSH}\left (\CC\right)|\ u\left (z\right)\leq \log\left (1+||z||\right)+C_u\rbr,$$
where the constant $C_u$ depends on the function $u$ but not on $z$. 

Let $K$ be a compact subset of $\CC$. The Green function of $K$ with pole at infinity, also known as the Siciak-Zahariuta extremal function (see \cite{Kli91} for more details) is defined by
\begin{equation*}
 V_K\left (z\right):=\sup\lbr v\left (z\right)|\ v\in \mathcal L\left (\CC\right), v|_K\leq 0\rbr.
\end{equation*}
This is a lower semicontinuous function in general and its upper semicontinuous regularization $V_K^{*}$ is defined by
\begin{equation}\label{VK}
V_K^{*}\left (z\right):=\limsup_{w\rightarrow z}V_K\left (w\right). 
\end{equation}

We recall the following classical theorem (its proof can be found, for example, in \cite{Kli91}): 

\begin{theorem} Let $K$ be a compact subset of $\CC$. Then $V_K^{*}\equiv+\infty$ if and only if $K$ is a pluripolar set. If $V_K^{*}\not\equiv+\infty$ then it is a plurisubharmonic function in the class $\mathcal L\left (\CC\right)$. Furthermore it is equal to zero on $K$ off a  (possibly empty) pluripolar set, and it is maximal outside $K$ in the sense that $\left (dd^cV_K^{*}\right)^n\equiv 0$ off $K$.
	\end{theorem}
	
	Of course in complex dimension one the last property means that $V_K^{*}$  is harmonic off $K$.
	
	The  {\it maximality} of  Green functions outside the set $K$   implies that they decay to zero as the argument approaches  the
 boundary of $K$ in the {\it slowest} possible fashion among all plurisubharmonic functions in the class $\mathcal L$, of
 course off the aforementioned pluripolar set.

\begin{definition}
 A compact set $K$ is called regular if $V_K$ is a continuous function.
\end{definition}
In particular for regular sets $V_K^{*}=V_K$ and $K=\{V_K^{*}=0\}$,  so the pluripolar set in the above theorem is empty. In our applications more regularity of $V_K^{*}$ will be needed:
\begin{definition}
 A regular compact set $K$ is said to have {\bf H\"older continuity property} of order $\alpha$ ($K\in HCP\left (\alpha\right)$) if the
 function $V_K=V_K^{*}$ is
$\alpha$-H\"older continuous.
\end{definition}
In fact it is enough to assume that $V_K$ is H\"older continuous at all points $z\in K$  (see \cite{BK11}).

Let us note that every compact and connected set in $\mathbb C$ has the H\"older continuity property of order $\frac{1}{2}$ (\cite{T}).

A condition partially converse to  (HCP) is the so-called \L ojasiewicz-Siciak condition:
\begin{definition}
 A regular compact set $K$ is said to satisfy the {\bf \L ojasiewicz-Siciak condition} of order $\alpha$  ($K\in LS\left (\alpha\right)$) if the function $V_K=V_K^{*}$
 satisfies the inequality
$$V_K\left (z\right)\geq Cdist\left (z,K\right)^{\alpha},\ {\rm if}\ dist\left (z,K\right)\leq 1$$
for some positive constant $C$ independent of the point $z$. The distance is with respect to the usual Euclidean metric.
\end{definition}
This notion was introduced by Gendre in \cite{Ge05} and was further studied by Bia\l as-Cie\.z and Kosek in \cite{BK11}  (see also \cite{Pi14}).

The following proposition will be crucial in what follows:
\begin{proposition}\label{HoldertoLS}
Let $K$ be a connected  compact set. Let also $g$ be the Riemann conformal map from $\mathbb C\setminus\mathbb D$ to the unbounded component of $\mathbb C\setminus K$, sending the infinity to infinity. If $g$ extends to the boundary as an $\alpha$-H\"older continuous mapping, then $K$ satisfies the \L ojasiewicz-Siciak condition of order $1/\alpha$. 
\end{proposition}
\begin{proof}
	The complement on the Riemann sphere of a connected compact set is simply connected and hence the Riemann mapping exists.
 Let $z\in \mathbb C\setminus K$ be a point satisfying $dist\left (z,K\right)\leq 1$. Let $w\in \mathbb C\setminus\mathbb D $ be the
 preimage of $z$
 under $g$. If $w_0$ is the closest point to $w$ lying on the unit circle then by assumption we obtain
\begin{equation*}
 C\left (|w|-1\right)^{\alpha}=C|w-w_0|^{\alpha}\geq|g\left (w\right)-g\left (w_0\right)|=|z-g\left (w_0\right)|.
\end{equation*}
If now $g^{-1}$ denotes the inverse mapping of $g$ we have
\begin{equation}
 |g^{-1}\left (z\right)|-1\geq\left (\frac{|z-g\left (w_0\right)|}{C}\right)^{1/\alpha}\geq \left (\frac{dist\left (z,K\right)}{C}\right)^{1/\alpha}
\end{equation}
The proof is finished by noticing that $V_K\left (z\right)=\log|g^{-1}\left (z\right)|$ in this case.
\end{proof}
\begin{remark}
	A stronger version of Carath\'eodeory theorem says that a further necessary condition for the assertion of Proposition
 \ref{HoldertoLS} to hold is that $\partial K$ should be locally connected.
\end{remark}
{\bf Quasiconformal mappings.}
The notion of a quasiconformal mapping is a generalization of the classical conformal maps. Below we present one of the equivalent definitions:
\begin{definition}
Let $f:U\rightarrow\Omega$ be a homeomorphism between domains in the complex plane. Now $f$ is said to be $K$-quasiconformal for some $K\geq 1$ if for any $z\in U$
$$\limsup_{r\rightarrow 0^{+}}\frac{max_{|h|=r}|f\left (z+h\right)-f\left (z\right)|}{min_{|h|=r}|f\left (z+h\right)-f\left (z\right)|}\leq K.$$
\end{definition}
$K$-quasiconformal mappings for $K=1$ are exactly the conformal ones. For $K>1$ these mappings are much more flexible, yet they share some of the basic properties of conformal maps.

The following is a classical regularity theorem for such maps:
\begin{theorem}
If $f:U\rightarrow\Omega$ is $K$-quasiconformal, then for any compact $E\Subset U$ $f|_E$ is a $1/K$-H\"older mapping with H\"older constant dependent only on $dist\left (E,\partial U\right)$.
\end{theorem}

The following corollary of this result will be used later on:

\begin{corollary}
 If $f$ is a conformal mapping from a domain $U$ onto a domain $\Omega$ which admits a $K$-quasiconformal extension to a  domain
 $U'$, such that $U\Subset U'$, then $f$ is $1/K$-H\"older continuous up to the boundary of $U$.
\end{corollary}

For more information on quasiconformal mappings one should consult \cite{Ah}.

\section{A remark on Mooney's argument in the real case}
In this section we shall briefly recall the existent theory of minimum sets for convex functions and the real Monge-Amp\`ere operator. An obvious generalization of the argument of Mooney from \cite{Mo15} yields a dependence of the additional regularity of the convex function and the dimension of the minimum set.
 
In \cite{Caf93} Caffarelli established the following bound for the Hausdorff dimension of the minimum set of a convex function:
\begin{theorem}
 Let $0\leq v\in\mathcal{CVX}\left (\Omega\right),\ \left (\Omega\Subset\mathbb R^n\right)$ satisfies $Det\left (D^2v\right)\geq C>0$  (the inequality is understood in the
 viscosity sense). Then
$$dim_{\mathcal H}\lbr v^{-1}\left (0\right)\rbr<\frac n2.$$
\end{theorem}
The proof makes implicit use of the local Lipschitz regularity of $v$.

In \cite{Mo15} Mooney gave a beautiful and completely elementary proof of the above fact which we sketch below. From this argument it is obvious that better dimension bounds are possible if more regularity on $v$ is assumed a priori.

Recall that a {\it section} associated to $v$ centered at $x$, supported by a subgradient vector $p$ and of height $h>0$ is the set defined by
$$S_{h,p}^v\left (x\right)=\lbr y\in\Omega| v\left (y\right)\leq v\left (x\right)+p\ {\color{red}\cdot}\left (y-x\right)+h\rbr.$$

As the graph of any convex function at any point is supported from below by a hyperplane (not necessarily unique), the existence of $p$ is guaranteed for any $x\in\Omega$. Of course if $v$ is differentiable at $x$, then $p=\nabla v\left (x\right)$. 

Mooney's argument hinges on the following proposition:
\begin{proposition}\label{sections}
 If $v$ solves $Det\left (D^2v\right)\geq1$ in a bounded domain $\Omega\in\mathbb R^n$ then for any section $S_{h,p}^v\left (x\right)$ one has the volume bound
$$|S_{h,p}^v\left (x\right)|\leq Ch^{n/2}.$$
\end{proposition}
\begin{proof}(Sketch)
Translating if necessary, one can assume that the center of mass of $S_{h,p}^v\left (x\right)$ is at the origin. Subtracting an affine function  (which does not change the Monge-Amp\`ere density) one can further assume that $p$ is the zero vector and $v$ is non positive with minimum equal to $-h$. By John's lemma, the convex set $S_{h,0}^v\left (x\right)$ can be transformed by a linear change of coordinates $A$ to a normalized convex set  (that is a convex subset containing the unit ball and contained in a larger concentric ball of fixed radius). Then a comparison of $u\left (Ax\right)$ with $||x||^2-1$ gives a bound for $|detA|$ and hence for the volume of $S_{h,0}^v\left (x\right)$.
 
\end{proof}

If now $v$ vanishes on a $k$-dimensional germ of a  hyperplane $L$ passing through the center of coordinates, then by Lipschitz regularity $v\left (x\right)\leq Cdist\left (x,L\right)$ and thus $S_{h,0}^v\left (0\right)$ grows at least as $\frac hC$ in the directions perpendicular to the plane. Hence its volume grows at least like $\frac{h}{C}^{n-k}$, contradicting for small $h$ Proposition \ref{sections} if $k>\frac{n}2$. The case $k=\frac{n}2$ can be handled by adding a linear function  (vanishing on $L$) to $v$ so that the one-sided growth $l\left (h\right)$ in one of the perpendicular directions is slower than $h$  (in the sense that $\lim_{h\rightarrow 0^{+}}\frac{l\left (h\right)}{h}=+\infty$).

Repeating this argument one immediately gets the following result which in particular covers the result of Urbas \cite{Ur88} on strict convexity:
\begin{proposition}Let $0\leq v\in\mathcal{CVX}\left (\Omega\right),\ \left (\Omega\Subset\mathbb R^n\right)$ satisfies $Det\left (D^2v\right)\geq C>0$ (the inequality is understood in the
 viscosity sense). Assume moreover that $v\in\mathcal C^{1,\alpha}$ for $\alpha>1-\frac{2k}n$. Then
$$dim_{\mathcal H}\lbr v^{-1}\left (0\right)\rbr< k.$$

In particular if $\alpha>1-2/n$, then $v$ is strictly convex in the sense that its graph does not contain affine germs.
\end{proposition}
\begin{proof}
 Let $L$ be the affine piece in $v^{-1}\left (0\right)$ of largest dimension. Exploiting that $v\left (x\right)\leq Cdist\left (x,L\right)^{1+\alpha}$ the section
 $S_{h,0}^v\left (0\right)$ in each direction perpendicular to $L$ has
 length no less than
 $\frac{h}{C}^{1/{\left (\alpha+1\right)}}$. Coupling this with Proposition \ref{sections} we get the inequality 
$$n-k\geq \frac{n\left (1+\alpha\right)}{2},$$
from which the result follows.
\end{proof}

\section{One dimensional case}
In complex dimension one Harvey-Wells theorem \cite{HW73} tells us that the minimum set of a smooth strictly subharmonic function is contained in a $\mathcal C^1$ smooth submanifold of $\mathbb C$. This is in fact trivial since $\Delta u\geq c>0$ implies that at any minimum point either $u_{xx}$ or $u_{yy}$ is nonzero, thus by the implicit function theorem one of the the sets $\lbr u_x=0\rbr,\ \lbr u_y=0\rbr$ is locally a graph of a $\mathcal C^1$ function. Our examples show that this argument fails dramatically if the smoothness assumption is dropped.

As already noted in the introduction, every regular compact set $K$ in the complex plane is the minimum set of a subharmonic function. When it comes to strictly subharmonic functions, one immediately sees that it must hold that $K=\partial K$, since a constant on $int K$ will fail to be strictly subharmonic. Also if we want global strictly subharmonic functions then trivially $K$ cannot disconnect the plane, since the maximum principle will force any bounded connected component of $\mathbb C\setminus K$ to belong to $K$. In particular no Jordan curve can be a minimum set of a global function and hence a direct converse of Harvey-Wells theorem fails, that is not every compact subset of a $\mathcal C^{1}$ smooth submanifold of $\mathbb C$ is a minimum set of a global strictly subharmonic function.    In order to make $K$ a minimum set of a strictly subharmonic function the basic idea is to perturb the function $V_K^{*}$ suitably. Heuristically the function $u_K:=\left (V_K^{*}\right)^2$ is ``more subharmonic'' with Laplace density equal to 

\begin{equation}\label{trick}
 \Delta u_K=2V_K^{*}\Delta V_K^{*}+2\left|\frac{\pa V_K^{*}}{\pa z}\right|^2,
\end{equation}
with the the first term vanishing as $K$ is a regular set. The nontrivial issue is then to establish a lower bound on $|\frac{\pa V_K^{*}}{\pa z}|$ {\it up to the boundary} $\partial K$. What matters is  the exact rate of convergence of $V_K$ to zero as $z\rightarrow z_0\in K$ (i.e. the exponent in the \L ojasiewicz-Siciak condition). Also it is important to rule out clustering
 of vanishing points for $\frac{\pa V_K^{*}}{\pa z}$ to the boundary of $K$.

Our first example shows that the minimum set can fail to be locally a graph:
\begin{example}
Let 
$$K:=\lbr re^{i\theta}|\ r\in[0,1], \theta=0,\frac{2\pi}{3}\ {\rm or}\ \frac{4\pi}{3}\rbr.$$
then the function $u_K:=\left (V_K^{*}\right)^{4/3}$ is strictly subharmonic in the sense that $\Delta u_k\geq c>0$, but $K=u_K^{-1}\left (0\right)$ is not a graph of a function around the  origin.
\end{example}
\begin{proof}
 
 We shall exploit the explicit formula for $V_K^{*}$ which can be obtained from the conformal map from $\mathbb C\setminus K$ to
 the complement of the unit disc  (see \cite{IP02})- it reads 
$$V_K^{*}\left (w\right)=V_K\left (w\right)=\frac13\log\left |2w^3-1+\sqrt{\left (2w^3-1\right)^2-1}\right|=\log|f\left (w\right)|$$
with the branch chosen so that $w\rightarrow \left (2w^3-1+\sqrt{\left (2w^3-1\right)^2-1}\right)^{1/3}=f\left (w\right)$ sends $\mathbb C\setminus K$ to the exterior of the unit disc. 

By computation $V_K\left (z\right)$ is Lipschitz at all points of $K$ except on the endpoints and $0$. At the endpoints it is $1/2$-H\"older, while $V_K=O\left (|w|^{3/2}\right)$ at the origin  (this also follows from \cite{Pi14}). Hence the exponent $4/3$ is the right one to prevent vanishing of the Laplace density at zero. Indeed
\begin{align*}
&\Delta u_K\left (w\right)=\frac1{9}[\log|f\left (w\right)|]^{-2/3}\frac{|\frac{\partial f}{\partial w}\left (w\right)|^2}{|f\left (w\right)|^2}
&=\frac1{9}[\log|f\left (w\right)|]^{-2/3}\frac{|w|^4|f(w)|^{2/3}}{|\left (w^3-1\right)w^3||f\left (w\right)|^2}.
\end{align*}
This quantity is obviously nowhere vanishing, and it is bounded below by a positive constant around the origin, by the exact asymptotics of $|f\left (w\right)|$. Also, by direct calculation, $f\in L^p$ for any $p<3/2$. 

\end{proof}
On the other hand a slight modification of the example above cannot be a minimum set of a strictly subharmonic function:
\begin{example}\label{five} Let
$$K:=\lbr re^{i\theta}|\ r\in[0,1], \theta=0,\frac{2\pi}{5}\,\ \frac{4\pi}{5},\ \frac{6\pi}{5} {\rm or}\ \frac{8\pi}{5}\rbr.$$
 Then there is no strictly subharmonic function in a neighborhood of $K$ which is nonnegative and vanishes on $K$.
\end{example}
\begin{proof}
 Again by we have the explicit formula
$$V_K\left (w\right)=V_K^{*}\left (w\right)=\frac15\log\left |2w^5-1+\sqrt{\left (2w^5-1\right)^2-1}\right|.$$
In particular $V_K=O\left (|w|^{5/2}\right)$ at the origin.

Suppose now that $u$ is a nonnegative strictly subharmonic function vanishing on $K$. Assume without loss of generality that $\Delta u\geq1$. Fix a neighborhood $U$ of $K$, such that $u$ is bounded from above on $\bar{U}$. Then by maximality of $V_K$ in $U\setminus K$ one has $V_K\geq cu$ for a sufficiently small positive constant $c$.

Fix a small positive radius $r$, such that the disc $B\left (0,r\right)\subset U$. Then by the Jensen formula we obtain
\begin{align*}
 &\frac1{2\pi}\int_0^{2\pi}u\left (re^{i\theta}\right)d\theta=\frac1{2\pi}\int_0^{2\pi}u\left (re^{i\theta}\right)d\theta-u\left (0\right)=
\frac1{2\pi}\int_0^rs^{-1}\int_{B\left (0,s\right)}\Delta u\left (z\right)dzds\\
&\geq \frac1{2\pi}\int_0^rs^{-1}\pi s^2ds=r^2/4.
\end{align*}
 But on the other hand

$$\frac1{2\pi}\int_0^{2\pi}u\left (re^{i\theta}\right)d\theta\leq \frac1{2\pi c}\int_0^{2\pi}V_K\left (re^{i\theta}\right)d\theta\leq Cr^{5/2}.$$
Coupling both estimates we get a contradiction for $r$ small enough.
\end{proof}
\begin{remark}
 The explicit formulas for the Green functions for the above sets were found by R. Pierzcha\l a (compare example 5.2 in \cite{Pi16}). We
 would like to thank him for
 pointing out this reference.
\end{remark}

The examples above suggest that $V_K$ should converge to zero not faster than quadratically for any $w\in\partial K$ i.e. the \L ojasiewicz-Siciak exponent should not be larger than 2. If such is the case an application of analogous idea to more general sets $K$ results in an abundance of examples. The following theorem summarizes what can be gotten
 by this construction:
\begin{theorem}\label{1dim}
 Let $K$ be a compact set with empty interior satisfying $LS\left (\alpha\right)$ for $\alpha<2$ (this implicitly rules out polar or
 non-regular sets). Then
\begin{enumerate}
\item If $K$ is connected and does not disconnect the plane  then
 it is a minimum set of a strictly subharmonic function;
\item If $K$ is porous then it is a minimum set of a function strictly subharmonic at $K$;
\end{enumerate} 
If in turn for some point $w\in K$ one has $V_K\left (z\right)=O\left (|z-w|^{\alpha}\right)$ for
 $\alpha>2$ then $K$ cannot be a minimum set of a strictly subharmonic function.
\end{theorem}
\begin{proof}
We start with $(1)$.

 Indeed, as $\hat{ \mathbb C}\setminus K$ is simply connected one can use the following estimate from \cite{GS}
 \begin{equation}
 \label{greenest}
 \frac{\sinh V_{K}^{*}(w)}{4\left|\frac{\partial V_{K}^{*}}{\partial w}\right|}\leq dist (w,K)\leq \frac{\sinh
 V_{K}^{*}(w)}{\left|\frac{\partial V_{K}^{*}}{\partial w}\right|}.\end{equation}
 		In particular $\frac{\pa V_K^{*}}{\pa w}$ never vanishes on $\mathbb C\setminus K$.

Let $u_K:=\left (V_K^{*}\right)^{2/\alpha}$. Then


 $$\Delta u_{K}=\frac{2}{\alpha}(V_K^{*})^{2/\alpha-1}\Delta V_{K}^{*}+\frac{2}{\alpha}\left(\frac{2}{\alpha}-1\right)(V_K^{*})^{2/\alpha-2}\left|\frac{\partial V_{K}^{*}}{\partial w}\right|^{2}$$$$=\frac{2}{\alpha}\left(\frac{2}{\alpha}-1\right)(V_K^{*})^{2/\alpha-2}\left|\frac{\partial V_{K}^{*}}{\partial w}\right|^{2}.$$
By \ref{greenest} this behaves like $$\frac{2}{\alpha}\left(\frac{2}{\alpha}-1\right)(V_K^{*})^{2/\alpha-2}\left|\frac{\sinh V_{K}^{*}}{dist(w,K)}\right|^{2}\sim\frac{2}{\alpha}\left(\frac{2}{\alpha}-1\right)\frac{(V_K^{*})^{2/\alpha}}{\left|dist(w,K)\right|^{2}},$$
	Which is bounded below by the \L ojasiewicz - Siciak condition.
	
Consider now the case when $K$ is porous. Let $u_K$ be as above and note that by assumption $u_K\left (w\right)\geq Cdist\left (w,K\right)^2$. On the other hand at every point $w_0\in K$ we have 
$$\Delta u_K\left (w_0\right)=\frac8{\pi}\lim_{r\rightarrow 0^{+}}\frac{\int_{B\left (w_0,r\right)}u_K}{r^4}.$$ 

As $K$ is porous there is a constant $\la$, $0<\la<1/2$, such that for any $r>0$ there is a point $w_1=w_1\left (r\right)\in B\left (w_0,r\right)$ such that the disc $B\left (w_1,\la r\right)$ belongs to $B\left (w_0,r\right)\setminus K$. But then for any $y\in B\left (w_1,\la r/2\right)$ the distance between $y$ and $K$ is at least $\la r/2$ and thus $u_K\left (y\right)\geq Cr^2$, by the \L ojasiewicz-Siciak condition on $V_K$. Thus
$$\liminf_{r\rightarrow 0^{+}}1/r^2\int_{B\left (w_0,r\right)}\Delta u_K\geq C' \liminf_{r\rightarrow 0^{+}}\frac{\int_{B\left (w_1,\la r/2\right)}u_K}{r^4}\geq c$$
for some positive constant $c$ dependent only on $\la$.

Finally the last statement follows from exactly the same reasoning as in Example \ref{five}.
\end{proof}
\begin{remark}
 In general $\frac{\pa V_K^{*}}{\pa w}$ vanishes somewhere away from $K$ if $K$ is not connected as the example of \newline
 $K=\cup_{j=1}^4\overline{B\left (i^j,1/2\right)}$ and the
 point
 $w=0$ shows. If one can control the distance between the set where $|\frac{\pa V_K^{*}}{\pa w}|$ is small and $K$ then $u_K$ will be
 strictly subharmonic in some neighborhood of $K$. This is always true if $K$ has finitely many components (actually $\frac{\pa
 V_K^{*}}{\pa w}$ will have exactly $k-1$ zeros, where $k$ is the number of components of $K$ which do not reduce to points, see
 \cite{So08}) and always wrong if there are infinitely many such components (since then the zeros of $\frac{\pa V_K^{*}}{\pa w}$ cluster
 on $K$, see \cite{So08}).
\end{remark}
\begin{corollary}
 Any compact regular subset $K$ of the real line is a minimum set for a function strictly subharmonic at $K$.
\end{corollary}
This is because any such set satisfies the \L ojasiewicz-Siciak condition with exponent $1$  (see \cite{Pi14}) and is obviously porous. 

More importantly the criterion is strong enough to produce minimum sets with Hausdorff dimension larger than one:
\begin{example}\label{julia} Let $J_{\la}$ be the Julia set of the polynomial $f_{\la}\left (z\right)=z^2+\la z$, $|\la|<1$. Then for $\la$ sufficiently close to zero $J_{\la}$ is a minimum set of  a strictly subharmonic function. The Hausdorff dimension of $J_{\la}$ satisfies $dim_{\mathcal H}J_{\la}\geq 1+0.36|\la|^2.$ 
\end{example}
\begin{proof}
 We follow closely the argument in \cite{BP87} Theorem B. In particular it is well known that for small $\la$ the Julia set is connected
 and its complement consists of two simply connected domains. As in \cite{BP87} we note that the conformal map $g_{\la}$ from the
 complement of the
 unit disc to the unbounded component $U$ of $\mathbb C\setminus J_{\la}$ admits a $K$-quasiconformal extension  (denoted by
 $\tilde{g}_{\la}$)
 to
 the whole of $\mathbb C$ for
 $\frac{K-1}{K+1}=|\la|$. In particular the conformal map $g_{\la}$ is $1/K$-H\"older continuous up to the boundary, and if $K<2$
 Proposition \ref{HoldertoLS} implies that $V_{J_{\la}}=\log|g_{\la}^{-1}|$ satisfies $LS\left (\al\right)$ for $\al<2$. Thus by Theorem
 \ref{1dim}
 there is a perturbation
 $\tilde{V}_{J_{\la}}$ which is strictly subharmonic, nonnegative and vanishing continuously at the boundary.

In order to complete the proof we need to ``fill in'' the bounded component of $\mathbb C\setminus J_{\la}$. To this end note that if $h_{\la}$ is the conformal map from the unit disc to this component  (normalized by fixing zero) then the quasiconformal reflection
$$\tilde{h}_{\la}\left (z\right)=\begin{cases} h\left (z\right)\ {\rm for}\ |z|\leq 1\\
\tilde{g}_{\la}\left (1/\overline{\tilde{g}^{-1}_{\la}\left (h\left (1/\overline{z}\right)\right)}\right)\ {\rm for}\ |z|>1                   
                  \end{cases}$$
is a $K^2$-quasiconformal mapping, hence it is $1/K^2$-H\"older continuous. Taking the Green function $G\left (z,0\right)$ with pole at zero we can apply the same reasoning  (away from 0) for $-G$ as for the function $V_K$  (note that $-G$ is still harmonic except at zero). Thus there is a strictly subharmonic function $\tilde{G}$ on the bounded component  (with small neighborhood of the origin deleted) which vanishes continuously on the boundary.

Finally the function

$$H:=\begin{cases}
  \tilde{V}_{J_{\la}}\left (z\right)\ {\rm if}\  z\in U\\
\tilde{G}\ {\rm if} z\in \mathbb C\setminus\left (J_{\la}\cup U\cup\lbr0\rbr\right) \\
0\ {\rm if}\ z\in J_{\la}
    \end{cases}
$$
satisfies all the requirements. 
\end{proof}
The function $\tilde{V}_{J_{\la}}$ solves the free boundary problem
\begin{equation}
 \begin{cases}
  \tilde{V}_{J_{\la}}\in SH\left (U\right)\\
\Delta \tilde{V}_{J_{\la}}\geq c>0\\
\lim_{z\rightarrow\partial U}\tilde{V}_{J_{\la}}=0
 \end{cases}
\end{equation}

Thus $J_{\la}$ is an example of a free boundary of Hausdorff dimension {\it larger than one} (to get the continuous vanishing of $\nabla \tilde{V}_{J_{\la}}$ note that $g_{\la}^{-1}$ is also H\"older continuous with H\"older exponents tending to $1$ as $|\la|$ goes to zero). By classical Caffarelli theorem \cite{Caf81} the free boundary is always of dimension less than $1$ if the Laplacian is Lipshitz  (and by our remark in the introduction it is enough to have the Laplacian uniformly bounded and in $W^{1,1}$). It can be checked that in our examples the Laplacians are in $W^{1,1-\varepsilon}$ for $\varepsilon$ dependent on $|\la|$, but not in $W^{1,1}$, and they are $L^p$ integrable with $p$ tending to infinity as $|\la|$ goes to zero but they are not in $L^{\infty}$. 

On the other side definite upper bounds on the Hausdorff dimension can be obtained in the case when $\Delta u\in L^{\infty}$. This can be established by proving {\it porosity} of the minimum set. The argument is classical and is standard in free boundary literature  (compare \cite{Bla01,KKPS00,LS03}) but we were unable to find the exact potential theoretic reference. Thus we reproduce the details for the sake of completeness.

\begin{theorem}
 Let $u\in\mathcal{SH}\left (\Omega\right)$ satisfies $0<c\leq\Delta u\leq C$. Then $u^{-1}\left (0\right)\cap K$ for any compact subset
 $K$
 is $\lambda$-porous
 with porosity constant dependent on
 $c$, $C$ and $dist\left (K,\partial\Omega\right)$.
\end{theorem}
\begin{proof}
Fix a disc $B\left (x_0,R\right)$ such that $B\left (x_0,2R\right)\Subset\Omega$ and $u\left (x_0\right)=0$. Since $u-c{|z-x_0|^2}$ is subharmonic, there is a point $x_1\in\partial B\left (x_0,R\right)$, such that $u\left (x_1\right)\geq c|x_1-x_0|^2$. 

Next we prove that $u\left (y\right)\leq Ddist\left (y,u^{-1}\left (0\right)\right)^2$ for some $D$ dependent merely on $c$, $C$ and the distance to the boundary of $\Omega$. The argument in fact implies that the solution is $\mathcal C^{1,1}$ {\it at the minimum points}.

To this end we shall exploit Riesz representation coupled with Harnack inequality.

It is enough to prove the estimate when the distance in question is sufficiently small for otherwise the estimate follows from the uniform bound on $u$.

Fix the point $y\in K$ and let $y_1$ be the closest point from $u^{-1}\left (0\right)\cap K$ to $y$ (if it is not unique choose any). For simplicity we may assume that $y_{1}=0$.  Let $dist\left (0,y\right)=r$. We can assume that $r$ is so small that $B\left (y_1,2r\right)\Subset\Omega$. Consider the disc $B\left (0,2r\right)$ and apply the Riesz representation to $u$ on it. We obtain
\begin{equation}\label{Riesz}
 u\left (y\right)=\frac1{2\pi}\int_{0}^{2\pi}\frac{4r^2-|y|^2}{|2re^{i\theta}-y|^2}u\left (2re^{i\theta}\right)d\theta+\int_{B\left
 (0,2r\right)}\log\left (\frac{2r|z-y|}{|4r^2-z\bar{
y}|}\right)
\Delta u\left (z\right)
\end{equation}
\begin{align*}
&\leq \frac{3r^2}{r^2}\frac1{2\pi}\int_{0}^{2\pi}u\left (2re^{i\theta}\right)d\theta+\int_{B\left (0,2r\right)}\log\left (\frac{2r|z-y|}{|4r^2-z\bar{y}|}\right)
c\\
&=3\left (u\left (0\right)-\int_{B\left (0,2r\right)}\log\left (\frac{|z|}2r\right)\Delta u\left (z\right)\right)-c\left (|y|^2-4r^2\right)\\
&\leq 3\left (C4r^2\right)-3cr^2=\left (12C+3c\right)r^2.
\end{align*}
where we have made use of the nonnegativity of $u$ and negativity of the Green function in the first inequality. Second and third inequalities follow in turn from Riesz representations of the functions $c\left (|z|^2-4r^2\right)$.

Exploiting both bounds we get that in $B\left (x_0,2R\right)$ there is a point $x_1$ at distance at least $\sqrt{\frac{c}{12C+3c}}R$ from $u^{-1}\left (0\right)$ which establishes the claimed porosity property.
\end{proof}
\begin{remark}
 An example of Blank \cite{Bla01} shows that in the case of bounded strictly positive right hand side the free boundary may spiral
 infinitely many times at
 points. In this example as well as in all examples that we are aware of the Hausdorff dimension is equal to one. It would be interesting
 to know whether this is true in general.
\end{remark}


\section{Multidimensional case}
 Example \ref{julia} can be immediately generalized to the multidimensional setting so that the minimum set is of Hausdorff dimension {\it
 larger than n}:
\begin{example}\label{juliamanydim} If $H\left (z\right)$ is the function from Example \ref{julia}, then the function
$$\tilde{H}\left (z_1,\cdots,z_n\right):=H\left (z_1\right)+\cdots +H\left (z_n\right)$$
satisfies $\left (dd^c\tilde{H}\right)^n\geq c>0$ whereas its minimum set is equal to the $n$-times Cartesian product of $J_{\la}$. 
\end{example}
In fact it is easy to construct minimum sets of even larger Hausdorff dimension (see Example \ref{generalizedpogorelov} below), but it should be emphasized that in this construction the minimum set does not contain nontrivial analytic subsets.

Our next result states that the dimension of an analytic set contained in the minimum set is controlled by the regularity of the function $u$ in the H\"older scale. In the proof we shall exploit an old idea of Urbas \cite{Ur88} with suitable modifications.
\begin{theorem}\label{thm1}
 Let $u\geq 0$ be a plurisubharmonic function  satisfying $\left (dd^c u\right)^n\geq 1$. If additionally $u\in \mathcal C^{1,\alpha}$ for
 $\alpha>1-\frac{2k}n$ if $2k\leq n$ or $u\in \mathcal C^{0,\beta}$ for $\beta> 2-\frac{2k}n$ if $2k>n$, then no
 analytic set of dimension $\geq k$ can be contained in $u^{-1}\left (0\right)$.
\end{theorem}
\begin{proof}

We shall deal with  both cases simultaneously writing $\beta=1+\alpha$ if necessary - this will not affect the argument.
 Suppose on contrary that $A$ is a $k$ dimensional analytic subset of $u^{-1}\left (0\right)$. Our goal will be to construct a barrier $v$ on a thin
 domain
 close to a
  (modification of) $A$ which will contradict the regularity that $u$ has.

Pick a point $x_0$ in the regular part of $A$. Then there is a biholomorphic mapping $\pi:U\rightarrow V$ of an open ball $U$ in $\CC$ to a neighborhood $V$ of $x_0$, such that 
$$\pi^{-1}\left (A\cap V\right)=\lbr z\in U| z_1=0,\cdots,z_{n-k}=0\rbr,$$
with $\left (z_1,\cdots,z_n\right)$ being the coordinates in $U$ centered at $0=\pi^{-1}\left (x_0\right)$. We can also assume that the Jacobian of $\pi$ at zero is equal to $1$.

Consider now the function $\tilde{u}\left (z\right):=u\left (\pi\left (z\right)\right)$. Then 
\begin{equation}\label{mautilde}
\left (dd^c\tilde{u}\left (z\right)\right)^n=\left (dd^cu\right)^n|_{\pi\left (z\right)}|Jac_{\pi}\left (z\right)|^2\geq 1/2,
\end{equation}
where $Jac_{\pi}$ stands for the (complex) Jacobian of the mapping $\pi$ and the last inequality follows by the smoothness of $Jac_{\pi}$  (we can shrink $U$ further if necessary). Denote by $M$ the $\alpha$-H\"older constant for $\nabla \tilde{u}$, which can be made as close to the H\"older constant of $\nabla u$ as necessary if $U$ is further shrunk.

Let now $z'=\left (z_1,\cdots,z_{n-k}\right), z''=\left (z_{n-k+1},\cdots,z_n\right)$ Then
\begin{equation}\label{regularityest}
\tilde{u}\left (z',z''\right)\leq \tiu\left (0,z''\right)+M||z'||^{1+\alpha}\leq A||z'||^2+A^{-\gamma}C_0, 
\end{equation}

with $\gamma=\frac{1+\al}{1-\al}$, $A$- any large positive constant and $C_0=M^{\frac2{1-\al}}\left (\left (\frac{1+\al}{2}\right)^{\frac{1+\al}{1-\al}}-\left (\frac{1+\al}{2}\right)^{\frac{2}{1-\al}}\right)$ (recall $\alpha<1$ in our convention).

Consider now the polydisc 
$$W:=\lbr z\in U|\ ||z'||\leq \rho, |z_{n-k+1}|\leq\rho,\cdots,|z_n|\leq\rho\rbr.$$
If $\rho$ is taken small enough, then $W\Subset U$. Fix such $\rho$ and
consider the barrier function $$w\left (z\right):=A||z'||^2+A^{-\gamma}C_0+\sum_{j=n-k+1}^n\frac{\ep}{\rho}\left (n\rho-Re\left (z_j\right)\right)+B\sum_{j=n-k+1}^n\left (|z_j|^2-\rho Re\left (z_j\right)\right),$$
with $0<B\leq 1$ and $\ep<<1$ to be chosen later on.

Note that if $||z'||=\rho$ then $v\geq A\rho^2+A^{-\gamma}C_0+k[\left (n-1\right)\ep-B\frac{\rho^2}{4}]$. Thus if $A\geq \sup_U\tiu+\frac{k}4$ we get $w\geq u$. 

On the other hand if for some $n\geq j\geq n-k+1$ we have $|z_j|=\rho$ then
\begin{equation}\label{boundary1}
 w\left (z',z''\right)\geq A|z'|^2+A^{-\gamma}C_0+k\left (n-1\right)\ep-\left (k-1\right)B\frac{\rho^2}4.
\end{equation}
Fixing $\ep=\frac{k-1}{\left (n-1\right)k}B\frac{\rho^2}{4}$ if $k>1$ and a small multiple of $B\frac{\rho^2}{4}$ if $k=1$ (if $\rho$ is small and $B\leq 1$ this quantity is clearly small) and exploiting (\ref{regularityest}) we again obtain $w\geq u$. 

If one can prove that $\left (dd^cw\right)^n\leq \frac12\leq \left(dd^c\tiu\right)^n$ then by comparison principle it would follow that $w\geq u$ over the whole polydisc. Note that $\left (dd^cw\right)^n=A^{n-k}B^{k}$, hence the choice $B=\left (\frac{1}{2A^{n-k}}\right)^{1/k}$  (if $A$ is large enough this is clearly less than one) satisfies this requirement.

Under such a choice of constants we obtain
$$0\leq \tiu\left (0',\rho/2,\cdots,\rho/2\right)\leq w\left (0',\rho/2,\cdots,\rho/2\right)=A^{-\gamma}C_0+k\left (n-\frac12\right)\ep-kB\frac{\rho^2}{4}.$$

We claim that the sum of the last two terms is negative. Indeed this is the case for $k=1$ and for $k>1$ we obtain
\begin{equation*}
k\left (n-\frac12\right)\ep-kB\frac{\rho^2}{4}=\left (\frac{\left (k-1\right)\left (n-1/2\right)}{n-1}-k\right)B\frac{\rho^2}{4}, 
\end{equation*}
by our choice of $\ep$, and the latter quantity is equal to $-\frac{2n-k-1}{2\left (n-1\right)}B\frac{\rho^2}4$. Comparing this with the first term above we end up with

$$0\leq A^{-\gamma}C_0-A^{-\frac{n-k}{k}}C_1\frac{\rho^2}4$$
for some numerical constant $C_1$. This must hold (for fixed small $\rho$) for every sufficiently large constant $A$, thus implying $\frac{n-k}{k}\geq\gamma$. This in turn reads
$$\alpha\leq 1-\frac{2k}n,$$
which is a contradiction.
\end{proof}

The following examples, slightly generalizing Pogorelov ones \cite{Pog71,Bl99}, show that the obtained exponents are sharp:
\begin{example}\label{generalizedpogorelov}
 Set $z=\left (z',z''\right)$ with $z'=\left (z_1,\cdots, z_{n-k}\right)$, $z''=\left (z_{n-k+1},\cdots,z_n\right)$ then the
 plurisubharmonic function
$$u_k\left (z\right):=||z'||^{2-\frac{2k}{n}}\left (1+||z''||^2\right)$$
has Monge-Amp\`ere density equal to 
$$\left (\frac{n-k}{n}\right)^2\left (1+||z''||^2\right)^{n-k-1},$$
which is strictly positive, but the minimum set contains the $k$-dimensional subspace $z'=0$.
\end{example}

\begin{remark}
 The result shows that if $u\in\mathcal C^{1,\alpha}$ for $\alpha>1-\frac2n$ then, in fact, the minimum set cannot contain analytic subsets
 of positive dimension. This is the complex analogue of a real result of Urbas \cite{Ur88} stating that convex solutions with regularity
 slightly better than in Pogorelov examples must be strictly convex.
\end{remark}

Our final result rules out analytic sets of suitable dimension on which plurisubharmonic functions are pluriharmonic. In fact, by a simple observation we show that this setting is not different than the one for minimum sets.
\begin{theorem}\label{plurih}
  Let $u$ be a plurisubharmonic function  satisfying $\left (dd^c u\right)^n\geq 1$. If additionally $u\in \mathcal C^{1,\alpha}$ for
 $\alpha>1-\frac{2k}n$ if $2k\leq n$ or $u\in \mathcal C^{0,\beta}$ for $\beta> 2-\frac{2k}n$ if $2k>n$, then for any
 analytic set $A$ of dimension greater than or equal to $k$, the function $u$ restricted to $A$ cannot be pluriharmonic.
\end{theorem}

\begin{remark}
Example \ref{generalizedpogorelov} clearly shows that the regularity assumptions are sharp.

\begin{proof}
 We shall once again follow the argument of Urbas \cite{Ur88} but with a twist. Arguing just like in the proof of Theorem \ref{thm1}, we
 can assume that the analytic set is given locally around the coordinate origin by 
$$\lbr z\ | z_1=0,\cdots,z_{n-k}=0\rbr.$$

Fix a small enough radius $\rho>0$ such that everyting is compactly supported in the domain of definition of $u$. Define the  symmetrization function $\tilde{u}$ by
$$\tilde{u}\left (z\right):=\frac{1}{\left (2\pi\right)^n}\int_0^{2\pi}\cdots\int_{0}^{2\pi}u\left (z_1e^{i\theta_1},\cdots,z_ne^{i\theta_n}\right)d\theta_1\cdots d\theta_n.$$
By definition $\tiu$ is plurisubharmonic in a neighborhood of the origin.
Note that by pluriharmonicity assumption we have 
\begin{equation}\label{abc}
\tiu\left (z',0''\right)=u\left (0',0''\right),
\end{equation}
while by the plurisubharmonicity of $u$ we get
\begin{equation}\label{def}
 \tiu\left (z',z''\right)\geq u\left (0',0''\right).
\end{equation}
Thus, adding a constant if necessary, we can assume that $\tiu\geq 0$ and $\tiu\left (0\right)=0$. Of course $\tiu$ is at least as regular as $u$. Note finally that
$$\left (dd^c\tiu\right)^n\geq1,$$
since $\tiu$ is a convex combination of $z\rightarrow u\left (ze^{i\theta}\right)$ and each such function has Monge-Amp\`ere density at least one  (strictly speaking one has to pass through discretization of the integrals and to apply the inequality for mixed Monge-Amp\`ere measures for the mixed terms from \cite{Di09} to get this inequality).

Finally an application of Theorem \ref{thm1} for the function $\tiu$ yields a contradiction.

\end{proof}
 
\end{remark}

\medskip
{\bf Acknowledgments}
The first named author was partially supported by NCN grant no 2013/11/D/ST1/02599. The second named auhor was partially supported by NCN grant no 2013/08/A/ST1/00312.


\begin{thebibliography}{KKPS01}
	\bibitem[Ah06]{Ah}L.Ahlfors: Lectures on Quasiconformal Mappings. Second edition. With supplemental chapters by C. J. Earle, I. Kra, M. Shishikura and J. H. Hubbard. University Lecture Series, {\bf38}. American Mathematical Society, Providence, RI, 2006. viii+162 pp. ISBN: 0-8218-3644-7  
\bibitem[Bla01]{Bla01} I.Blank: Sharp results for the regularity and stability of the free boundary in the obstacle problem. Indiana Univ. Math. J.  {\bf50}  (2001),  no. 3, 1077-1112.
\bibitem[Bl99]{Bl99}  Z.B\l ocki:  On the regularity of the
complex Monge-Amp\`ere operator, Complex Geometric Analysis
in Pohang, 1997, 181-189  Contemp. Math., {\bf222}, Amer.
Math. Soc., Providence, RI, 1999.
\bibitem[BD11]{BD11} Z.B\l ocki, S. Dinew: A local regularity of the complex Monge-Amp\`ere equation. Math. Ann.  {\bf 351}  (2011),  no. 2, 411-416.
\bibitem[BK11]{BK11}L.Bia\l as-Cie\.z, M. Kosek:
Iterated function systems and \L ojasiewicz-Siciak condition of Green's function.  
Potential Anal.  {\bf 34}  (2011),  no. 3, 207-221.
\bibitem[BP87]{BP87} J.Becker, Ch. Pommerenke:
On the Hausdorff dimension of quasicircles. 
Ann. Acad. Sci. Fenn. Ser. A I Math.  {\bf12}  (1987),  no. 2, 329-333.
\bibitem[Caf77]{Caf77} L.A.Caffarelli: The regularity of free boundaries in higher dimensions. Acta Math.  {\bf139}  (1977), no. 3-4, 155-184.
\bibitem[Caf81]{Caf81} L. A.Caffarelli:   A remark on the Hausdorff measure of a free boundary, and the convergence of coincidence sets. Boll. Un. Mat. Ital. A (5)  {\bf18}  (1981), no. 1, 109-113.
\bibitem[Caf89]{Caf89} L.A.Caffarelli: A localization property of viscosity solutions to the Monge-Amp\`ere equation and their strict convexity. Ann. of Math. (2)  {\bf131}  (1990),  no. 1, 129-134.
\bibitem[Caf90]{Caf90b} L.A.Caffarelli: Interior $W^{2,p}$ estimates for solutions of the Monge-Amp\`ere equation. Ann. of 
Math. (2)  {\bf131}   (1990),  no. 1, 135-150.
\bibitem[Caf93]{Caf93} L.A.Caffarelli: A note on the degeneracy of convex solutions to Monge Amp\`ere equation. Comm. Partial Differential Equations  {\bf18}  (1993),  no. 7-8, 1213-1217.
\bibitem[Cat84]{Cat84} D.W.Catlin: Global regularity of the $\overline{\partial}$-Neumann problem.  Complex analysis of several 
variables (Madison, Wis., 1982),  39-49, Proc. Sympos. Pure Math., {\bf41}, Amer. Math. Soc., Providence, RI, 1984.
\bibitem[Di09]{Di09} S.Dinew: An inequality for mixed Monge-Amp\`ere measures. Math. Z.  {\bf262}   (2009),  no. 1, 1-15.
\bibitem[DD]{DD} S.Dinew, \.Z.Dinew: Free boundary problem for the complex Monge Amp\`ere operator, in preparation
\bibitem[DZZ11]{DZZ11} S.Dinew, X.Zhang, X.Zhang: The $\mathcal C^{2,\alpha}$ estimate of complex Monge-Amp\`ere 
equation. Indiana Univ. Math. J.  {\bf60}  (2011),  no. 5, 1713-1722.
\bibitem[EM84]{EM84} H.El Mir: Sur le prolongement des courants positifs ferm\'es. (French) [On the extension of closed positive currents] Acta Math.  {\bf153} (1984),  no. 1-2, 1-45.
\bibitem[Ge05]{Ge05}L.Gendre: In\'egalit\'es de Markov singuli\`eres et approximation des fonctions holomorphes de la classe $M$. PhD thesis, 
Universit\'e Toulouse III Paul Sabatier  (2005).
\bibitem[GS12]{GS}B.Gustafsson, A.Sebbar: Critical points of Green's function and geometric function theory.   Indiana Univ. Math. J.  {\bf61} (2012),  939-1017.
\bibitem [Gut01]{Gut01} C.E. Gutierrez: The Monge-Amp\`ere equation. Progress in Nonlinear Differential Equations and their Applications,
 {\bf44}. Birkh\"user Boston, Inc., Boston, MA, 2001. 127 pp.
\bibitem[HW73]{HW73} F.R.Harvey, R.O.Wells Jr.: Zero sets of non-negative strictly plurisubharmonic functions. Math. Ann.  {\bf201}  (1973), 165-170.
\bibitem[IP02]{IP02} V.Ivanov, V.Popov: Conformal Mappings and Their Applications, (in Russian) Editorial URSS, Moscow 2002.
\bibitem[KKPS00]{KKPS00} L.Karp, T.Kilpel\"ainen, A.Petrosyan, H.Shahgholian: On the porosity of free boundaries in degenerate variational inequalities.  
J. Differential Equations {\bf 164}  (2000),  no. 1, 110-117.
\bibitem[Kli91]{Kli91} M.Klimek:
Pluripotential Theory. London Mathematical Society Monographs. New Series, {\bf6}. Oxford Science Publications. The Clarendon Press, Oxford University Press, New York  1991. xiv+266 pp. ISBN: 0-19-853568-6
\bibitem[LS03]{LS03} K.Lee, H.Shahgholian: Hausdorff measure and stability for the $p$-obstacle problem 
$\left (2<p<\infty\right)$. 
J. Differential Equations  {\bf195}  (2003),  no. 1, 14-24.
\bibitem[Ma95]{Ma95} P.Mattila: Geometry of Sets and Measures in Euclidean Spaces. Fractals and Rectifiability. Cambridge Studies in Advanced Mathematics, {\bf44}. Cambridge University Press, Cambridge, 1995. xii+343 pp. ISBN: 0-521-46576-1; 0-521-65595-1
\bibitem[Mo15]{Mo15}  C.Mooney: Partial regularity for singular solutions to the Monge-Amp\`ere equation. Comm. Pure Appl. Math.  {\bf68}   (2015),  no. 6, 1066-1084.
\bibitem[Pi14]{Pi14} R.Pierzcha\l a:  
The \L ojasiewicz-Siciak condition of the pluricomplex Green function.  
Potential Anal.  {\bf 40}  (2014),  no. 1, 41-56.
\bibitem[Pi16]{Pi16} R.Pierzcha\l a: \L ojasiewicz-type inequalities in complex analysis, J. Anal. Math. to appear.
\bibitem[Pog71]{Pog71}  A.V.Pogorelov: A regular solution of the $n$-dimensional Minkowski problem. Dokl. Akad. Nauk SSSR {\bf199}, 785-788  (Russian); translated in Soviet Math. Dokl. {\bf12} 1971, 1192-1196.
\bibitem[So08]{So08} A.Yu.Solynin: A note on equilibrium points of Green's function. Proc. Amer. Math. Soc. {\bf136}(2008), no. 3, 1019-1021.
\bibitem[To10]{T} V.Totik: Christoffel functions on curves and domains. Trans. Amer. Math. Soc. {\bf362} (2010), 2053-2087.
\bibitem[Ur88]{Ur88} J. I. E. Urbas:   Regularity of generalized solutions of Monge-Amp\`ere equations. Math. Z. {\bf 197}   (1988),  no. 3, 365-393.
\bibitem[Wa12]{Wa13} Y.Wang: On the $\mathcal C^{2,\alpha}$-regularity of the complex Monge-Amp\`ere equation. Math. Res. Lett.  {\bf 19} 
 (2012),  no. 4, 939-946. 
\end{thebibliography}
\end{document}